\documentclass{amsart}

\usepackage{amssymb,amsmath,graphicx}

\usepackage{color}

\newtoks\prt

\numberwithin{equation}{section}

\newtheorem{thm}{Theorem}[section]

\newtheorem{lemma}[thm]{Lemma}

\theoremstyle{definition}

\newtheorem{remark}[thm]{Remark}

\newtheorem{definition}[thm]{Definition}

\def\eqn#1$$#2$${\begin{equation}\label#1#2\end{equation}}

\def\A{\mathcal A}

\def\C{\mathcal C}

\def\E{\mathcal E}

\def\e{e^\ast}

\def\s{u^\ast}

\def\t{v^\ast}

\def\E{E^\ast}

\def\H{\mathcal H}

\def\M{\mathcal M}

\def\ep{\varepsilon}

\def\en{\mathbb N}

\def\er{\mathbb R}

\def\ov{\overline}

\def \Ch {\operatorname{Ch}}

\def \ext {\operatorname{ext}}

\def\wh{\widehat}

\def \reg {\partial _{\kern1pt\text{reg}}}

\def\la{\langle}

\def\ra{\rangle}

\newcommand{\norm}[1]{\left\|#1\right\|}

\newcommand{\abs}[1]{\left| #1  \right|}

\begin{document}

\title[An Amir-Cambern theorem for subspaces]
{An Amir-Cambern theorem for subspaces of Banach lattice-valued continuous functions}

\author{Jakub Rondo\v s and Ji\v r\'\i\ Spurn\'y}

\address{Charles University\\
Faculty of Mathematics and Physics\\
Department of Mathematical Analysis \\
Sokolovsk\'{a} 83, 186 \ 75\\Praha 8, Czech Republic}

\email{jakub.rondos@gmail.com}
\email{spurny@karlin.mff.cuni.cz}



\begin{abstract}
 For $i=1,2$, let $E_i$ be a reflexive Banach lattice over $\er$ with a certain parameter $\lambda^+(E_i)>1$,
 let $K_i$ be a locally compact (Hausdorff) topological space and let $\H_i$ be a closed subspace of $\C_0(K_i, E_i)$ such that each point of the Choquet boundary $\Ch_{\H_i} K_i$ of $\H_i$ is a weak peak point. We show that if there exists an isomorphism $T\colon \H_1\to \H_2$ with $\norm{T}\cdot \norm{T^{-1}}<\min \lbrace \lambda^+(E_1), \lambda^+(E_2) \rbrace$ such that $T$ and $T^{-1}$ preserve positivity, then $\Ch_{\H_1} K_1$ is homeomorphic to $\Ch_{\H_2} K_2$.
\end{abstract}

\maketitle

\section{Introduction}
We work within the framework of real Banach spaces and real Banach lattices. If $E$ is a real Banach space then $E^*$ stands for its dual space. We denote by $B_E$ and $S_E$ the unit ball and sphere in $E$, respectively, and we write $\la\cdot,\cdot\ra\colon E^*\times E\to\er$ for the duality mapping. For a locally compact (Hausdorff) space $K$, let $\C_0(K,E)$ denote the space of all continuous $E$-valued functions vanishing at infinity. We consider this space endowed with the sup-norm
\[\norm{f}_{\sup}=\sup_{x \in K} \norm{f(x)}, \quad f \in \C_0(K,E).\]
If $K$ is compact, then this space will be denoted by $\C(K,E)$. For a compact space $K$, we identify the dual space $(\C(K,E))^*$ with the space $\M(K,\E)$ of all $\E$-valued Radon measures on $K$ endowed with the variation norm via Singer's theorem (see \cite[p. 192]{Singer}). Thus $\M(K, \er)$ is the usual space of (signed) Radon measures on $K$. Unless otherwise stated, we consider $\M(K,\E)$ endowed with the weak$^*$ topology given by this duality.

The well-known Banach-Stone theorem asserts that, given a pair of compact spaces $K$ and $L$, they are homeomorphic provided $\C(K,\er)$ is isometric to $\C(L,\er)$ (see \cite[Theorem 3.117]{faspol}).

A nice generalization of this theorem was given independently by Amir \cite{amir} and Cambern \cite{cambern}. They showed that compact spaces $K$ and $L$ are homeomorphic if there exists an isomorphism $T\colon \C(K,\er)\to\C(L,\er)$ with $\norm{T}\cdot \norm{T^{-1}}<2$. Alternative proofs were given by Cohen \cite{cohen} and Drewnowski \cite{drewnow}.

In a recent extension of the theorem of Amir and Cambern to the context of vector-valued functions  \cite{cidralgalegovillamizar}, it was showed that if $E$ is a real or complex reflexive Banach space with $\lambda(E)>1$, then for all locally compact spaces $K_1, K_2$, the existence of an isomorphism $T:\C_0(K_1, E) \rightarrow \C_0(K_2, E)$ with $\norm{T}\cdot\norm{T^{-1}}<\lambda(E)$ implies that the spaces $K_1, K_2$ are homeomorphic. Here

\[\lambda(E)=\inf \lbrace \max \lbrace \Vert e_1+\lambda e_2 \Vert: \lambda \in \er, \abs{\lambda}=1 \rbrace: e_1, e_2 \in S_E \rbrace\]
is a parameter introduced by Jarosz in \cite{jarosz-pacific}.

Also, in \cite{cidralgalegovillamizar} the authors proved that the constant $\lambda(E)$ is the best possible for $E=l_p$, where $2 \leq p <\infty$.

In \cite{Villamizar_svazy} it was shown how the constant $\lambda(E)$ can be improved under aditional assumptions on the isomorphism $T$. More specifically, the authors assume that $E$ is a real Banach lattice and $T:\C_0(K_1, E) \rightarrow \C_0(K_2, E)$ is a Banach lattice isomorphism. The constant $\lambda(E)$ may be then replaced by 
\[\lambda^+(E)=\inf \{\max \{\norm{e_1+e_2}, \norm{e_1-e_2}\}, e_1, e_2 \in S_{E}, e_1, e_2 \geq 0\}.\]
It is easily seen that $\lambda(E) \leq \lambda^+(E)$ for each Banach lattice $E$, and in \cite{Villamizar_svazy} it is shown that for $E=\ell_p$, where $1 \leq p <2$, the inequality is strict, and moreover, that the constant $\lambda^+(E)$ is the best possible for classical spaces $E=\ell_p$, where $p \in [1, \infty)$. Also, from the definition it follows that $1 \leq \lambda^+(E) \leq 2$ for each Banach space $E$.

Our research aims to extend isomorphic Banach-Stone type theorems to the context of subspaces of continuous functions, whose Choquet boundaries consist of weak peak points. Those two notions will be described in the next section.

In \cite{rondos-spurny-spaces}, we were able to extend the theorem of Amir and Cambern by showing that for closed subspaces $\H_i \subset \C_0(K_i,\er)$ for $i=1,2$, their Choquet boundaries are homeomorphic provided points in the Choquet boundaries are weak peak points and there exists an isomorphism $T\colon \H_1\to \H_2$ with $\norm{T}\cdot \norm{T^{-1}}<2$. Before that, there were proved analougous results for spaces of affine real continuous functions on compact convex sets (\cite{chuco}, \cite{lusppams} and \cite{dosp}), and spaces of affine complex continuous functions \cite{rondos-spurny}.

In \cite{rondos-spurny-vector-spaces}, we generalized the results of \cite{cidralgalegovillamizar} by showing that if for $i=1,2$, $E_i$ is a reflexive real or complex Banach space and $\H_i$ is a closed linear subspace of $\C_0(K_i, E_i)$, such that each point of $\Ch_{\H_i} K_i$ is a weak peak point, then if $\H_1$ is isomorphic to $\H_2$ by an isomorphism $T$ satisfying $\norm{T}\norm{T^{-1}}<\min\lbrace \lambda(E_1), \lambda(E_2) \rbrace$, then the Choquet boundaries $\Ch_{\H_1} K_1$ and $\Ch_{\H_2} K_2$ are homeomorphic.

The aim of this paper is to present an analogue of the result of \cite{Villamizar_svazy} in the context of subspaces. However, we are able to extend this result to this general setting only in the case where the Banach lattices are reflexive, which is not needed in \cite{Villamizar_svazy}. The reason for this is that meanwhile operators from $\C(K, E_1)$ spaces to $E_2$ may be represented by Borel measures on $K$ with values in $L(E_1, E_2^{**})$, the space of bounded linear operators from $E_1$ to $E_2^{**}$ (see Remark \ref{remark}), such a representation is not available for operators from $\H_1 \subset \C(K, E_1)$ to $E_2$. On the other hand, we are able to replace the assumption that the isomorphism $T$ is a Banach lattice isomorphism by the weaker condition that $T$ and $T^{-1}$ preserve positive elements.

More specifically, our main result is the following theorem.

\begin{thm}
	\label{main}
	For $i=1,2$, let $\H_i$ be a closed subspace of $\C_0(K_i, E_i)$ for some locally compact space $K_i$ and a real reflexive Banach lattice $E_i$ satisfying $\lambda^+(E_i)>1$. Assume that each point of the Choquet boundary $\Ch_{\H_i} K_i$ of $\H_i$ is a weak peak point and let $T\colon \H_1\to \H_2$ be an isomorphism satisfying \[\norm{T}\cdot\norm{T^{-1}}<\min \lbrace \lambda^+(E_1), \lambda^+(E_2) \rbrace\] such that $T$ and $T^{-1}$ preserve positive elements, that is, 
	\[ f \in \H_1^+ \quad \text{if and only if} \quad T(f) \in \H_2^+.\]
	Then $\Ch_{\H_1} K_1$ is homeomorphic to $\Ch_{\H_2} K_2$.
\end{thm}

The proof of this theorem will be given in Section 5.

\section{Definitions and notation}

If $E$ is a real Banach lattice, then we write $E^+$ for the set of all positive elements of $E$. Then $E^*$ is also a Banach lattice with the ordering given by $e^* \geq 0$ in $E^*$ if and only if $\la e^*, e \ra \geq 0$ for each $e \in E^+$. It is well-known that if $e^*$ is a positive element of $\E$, then the norm of $\e$ is determined by its values on positive elements, that is, 
\[\norm{\e}=\sup_{e \in B_{E^+}} \abs{\la \e, e \ra},\]

see e.g. \cite[Theorem 4.1 and the subsequent equalities]{Aliprantis2006}.

If $K$ is a locally compact (Hausdorff) space, then the space $\C_0(K, E)$ is a Banach lattice with the natural pointwise ordering. If $\H$ is a linear subspace of $\C_0(K, E)$ (which is generally not a sublattice of $\C_0(K, E)$), then we say that a function $f \in \H$ is positive if 
$f(x) \in E^+$ for each $x \in K$, and we write $\H^+$ for the set of positive elements of $\H$. Also, we can naturally consider positivity on  $\H^*$ by saying that $s \in \H^*$ is positive (and we write $s \in (\H^*)^+$), if $s(h) \geq 0$ for each $h \in \H^+$. We also consider positivity on $\H^{**}$ in the same obvious way.

We will from now on tacitly assume that the dimension of both the spaces $E$ and $\H$ is at least 1. If $\H$ or $E$ has the dimension zero then the assumptions of our main results are never satisfied.

For $h \in \H$ and $\e \in \E$, $\e(h)$ is the element of $\C_0(K, \er)$ defined by $\e(h)(x)=\la \e, h(x) \ra$ for $x \in K$. As in \cite{rondos-spurny-vector-spaces}, we define the canonical scalar function space $\A \subset \C_0(K, \er)$ associated to $\H$ as the closed linear span of the set 
\[ \lbrace \e(h): \e \in \E, h \in \H \rbrace  \subset \C_0(K, \er).\]
Since both the spaces $\H$ and $E$ are of dimension at least 1 by the assumption, it follows that the dimension of $\A$ is at least 1 as well. 

The sets $B_{\E}$, $B_{\H^*}$ and $B_{\A^*}$ will be always equipped with the $w^*$-topology, unless otherwise stated.
We consider evaluation mappings $i$, $\phi$ defined as 
\[ i\colon K \to B_{\A^*}, x \mapsto i(x), \quad 
\phi\colon K \times \E \to \H^*, (x, \e) \mapsto \phi(x, \e), \]
where 
\[ \la i(x), a \ra=a(x), \quad a \in \A, \quad \text{and} \quad
\la \phi(x, \e), h \ra=\la \e, h(x) \ra=\e(h)(x), \quad h \in \H.\]
The mappings $i$ and $\phi$ are continuous, if on $\E$ we consider the $w^*$-topology. Moreover, it follows easily from the definition that $\phi$ is linear with respect to $\E$, $\phi(K \times B_{\E}) \subset B_{\H^*}$ and $\phi(K \times (E^*)^+) \subset (\H^*)^+$.

We define the \emph{Choquet boundary} $\Ch_{\H} K$ of $\H$ as the Choquet boundary of $\A$, that is, $\Ch_{\H} K$ is the set of those points $x \in K$ such that $i(x)$ is an extreme point of $B_{\A^{*}}$.

Next, for a function $f \colon K \rightarrow \er$ and $e \in E$, the function $f \otimes e \colon K \rightarrow E$ is defined by
\[ (f \otimes e)(x)=f(x)e, \quad x \in K .\]

\begin{definition}
	Let $\H$ be a closed subspace of $\C_0(K, E)$. We say that a point $x \in \Ch_{\H} K$ is a \emph{weak peak point}, if for each neighbourhood $U$ of $x$, $\ep \in (0, 1)$ and $e \in E^+$ there exists a function $h_{U, \ep} \in \A$, the canonical scalar function space of $\H$, such that $h_{U, \ep}(x)>1-\ep$, $h<\ep$ on $\Ch_{\H} K \setminus U$, $0 \leq h_{U, \ep} \leq 1$ on $K$ and $h_{U, \ep} \otimes e \in \H$.
\end{definition}

Note that the above definition of a weak peak point differs slightly from the one in \cite{rondos-spurny-vector-spaces}. However, if $K$ is compact and the space $\H$ contains a nonzero constant function, then so does its canonical scalar function space $\A$, and if $h \in \A$ is a peaking function in the sense of \cite{rondos-spurny-vector-spaces}, then $\frac{h+\ep}{1+\ep} \in \A$ is a peaking function in the sense of this paper. Thus in the case when $\H$ contains a nonzero constant function, the assumption on weak peak points that we use here is weaker that the one in \cite{rondos-spurny-vector-spaces}, as here we consider only positive elements of $E$.

The reason for this modification is that we need the peaking function $h$ to satisfy that $h \otimes u \in \H$ is positive whenever $u$ is positive in $E$. It readily follows that the conclusions of Lemmas 2.6, 2.7 and 2.10 in \cite{rondos-spurny-vector-spaces} about weak peak points, which we are going to use here as well, remain true with this slightly modified definition.

\begin{definition}
	Let $\H$ be a closed subspace of $\C_0(K, E)$. We consider an ordering $\prec$ on the pairs $(U, \ep)$, where $U$ is a nonempty closed set and $\ep>0$ by $(U_1, \ep_1)\prec (U_2, \ep_2)$ if $U_2 \subset U_1$ and $\ep_2 \leq \ep_1$. 
	
	If $x \in \Ch_{\H} K$ is a weak peak point and $e \in S_{E^+}$, then we define the \emph{net of peaking functions} for the pair $(x, e)$ as the system
	\[ \{h_{U, \ep}: U \text{ is a closed neighbourhood of }x \text{ and } \ep \in (0,1) \},  \] 
	where $h_{U, \ep}$ is a function in $\A$, the canonical scalar function space of $\H$, and satisfies that $h_{U, \ep}(x)>1-\ep$, $h_{U, \ep}<\ep$ on $\Ch_{\H} K \setminus U$, $0 \leq h_{U, \ep} \leq 1$ on $K$ and $h_{U, \ep} \otimes e \in \H$, endowed with the ordering $\prec$. Note that since $U$ is closed and $h_{U, \ep}$ is continuous, $h_{U, \ep} \leq \ep$ on the set
	\[\ov{\Ch_{\H} K}\setminus U \subset \ov{\Ch_{\H} K \setminus U}.\] 
\end{definition}

\section{Auxiliary results for the case of compact spaces}

In this section we assume that $K$ is a compact space and $E$ is a real Banach lattice, but mostly we actually use just the Banach space structure of $E$, except of the part concerning positive elements. We collect some well known facts about the spaces of $\C(K, E)$ and $\M(K, \E)$, as well as some of the auxiliary results for subspaces  $\H \subset \C(K, E)$ that we obtained in \cite{rondos-spurny-vector-spaces}, and that are necessary for the proof of our main result. The reduction from the case of locally compact spaces to the case of compact ones in the proof of Theorem \ref{main} will be possible due to \cite[Lemma 2.10]{rondos-spurny-vector-spaces}. 

To begin with, by \cite[Lemma 2.1]{rondos-spurny-vector-spaces} we know that
\begin{equation}
\label{ext}
\ext B_{\H^*}\subset \phi(\Ch_{\H} K \times \ext B_{\E}).
\end{equation}
Further, it was proved in \cite[Lemma 2.2]{rondos-spurny-vector-spaces} that for any $s\in\H^*$ there exists a vector measure $\mu\in\M(\ov{\Ch_{\H} K}, \E)$ such that $\mu=s$ on $\H$ and $\norm{\mu}=\norm{s}$.

Next we recall that, given a pair of topological spaces $M,L$, a function $f\colon M\to L$ is \emph{of the first Borel class} if $f^{-1}(U)$ is a countable union od differences of closed sets in $M$ for any $U\subset L$ open (see \cite{spurny-amh} or \cite[Definition~5.13]{lmns}). The following maximum principle is what makes this class of functions so important for us. If $f\colon X\to \er$ is a bounded affine function of the first Borel class on a compact convex set $X$, then
\[
\sup_{x\in X} \abs{f(x)}=\sup_{x\in \ext X}\abs{f(x)},
\]
see \cite[Corollary~1.5]{dosp} and \cite[Theorem~2.3]{koumou}. 

Consequently, if $F^{**} \in B_{\H^{**}}$ is a linear functional that is of the first Borel class on the compact convex set $B_{\H^*}$, then we have 
\begin{equation}
\nonumber
\begin{aligned}
&\Vert F^{**} \Vert=
\sup_{s \in \ext B_{\H^*}} \abs{\langle F^{**}, s \rangle} =^{\eqref{ext}} \sup_{x \in \Ch_{\H} K, \e \in S_{\E}} \abs{\langle F^{**}, \phi(x, \e) \rangle}.
\end{aligned}
\end{equation}

Let us now moreover suppose that $F^{**}$ is a positive element of $\H^{**}$. For a given $x \in K$, using the linearity of the evaluation mapping $\phi$ with respect to $\E$ one may define an element $F^{**}(x) \in E^{**}$ by the formula 
\[\la F^{**}(x), \e \ra=\la F^{**}, \phi(x, \e) \ra, \quad \e \in \E.\] 
We claim that this element is positive in $E^{**}$. Indeed, for a given $\e \in (\E)^+$, we know that $\phi(x, \e)$ is positive in $\H^*$, and thus
\begin{equation}
\label{positiveness}
\la F^{**}(x), \e \ra=\la F^{**}, \phi(x, \e) \ra \geq 0
\end{equation}
by the assumption. Thus using the previous equality we obtain
\begin{equation}
\label{max}
\begin{aligned}
&\Vert F^{**} \Vert=
 \sup_{x \in \Ch_{\H} K, \e \in S_{\E}} \abs{\langle F^{**}, \phi(x, \e) \rangle}= \sup_{x \in \Ch_{\H} K, \e \in S_{\E}} \abs{\langle F^{**}(x), \e \rangle}=\\&=
 \sup_{x \in \Ch_{\H} K, \e \in S_{(\E)^+}} \abs{\langle F^{**}(x), \e \rangle}=\sup_{x \in \Ch_{\H} K, \e \in S_{(\E)^+}} \abs{\langle F^{**}, \phi(x, \e) \rangle}.
\end{aligned}
\end{equation}

Next, if $f \in \C(K, \er)$ and $e \in E$, then $f \otimes e \in \C(K, E)$ with $\norm{f \otimes e}=\norm{f}\norm{e}$, and it follows from the form of duality between $\M(K, \E)$ and $\C(K, E)^*$ (see \cite[pages 192 and 193]{Singer}) that 
\begin{equation}
\label{application}
\langle \mu, f \otimes e \rangle=\la \mu, e \ra(f), \quad \mu \in \M(K, \E),
\end{equation}
where $\langle \mu, e \rangle \in M(K)$ is defined by 
\[ \la \mu, e \ra(A)=\la \mu(A), e \ra, \quad A \subset K \text{ Borel}.\]
Also if $f \colon K \rightarrow \er$ is a bounded Borel function, then for a vector measure $\mu \in \M(K, \E)$ and $e \in E$ we consider the application $\la \mu, f \otimes e \ra$ of $\mu$ on $f \otimes e$ given by \eqref{application}.

Further, if $\mu \in \M(K, \er)$ and $\e \in \E$, then the vector measure $\e \mu \in \M(K, \E)$ is defined by
\[  \la \e \mu, h \ra= \mu(\e(h)), \quad h \in \C(K, E).\]

If $f \colon K \rightarrow \er$ is a bounded Borel function, $\mu \in \M(K, \er)$, $\e \in \E$ and $e \in E$, then it holds that
\begin{equation}
\label{aplikace}
\la \e \mu, f \otimes e \ra=\la \e, e \ra \mu(f),
\end{equation}
see \cite[(2.2)]{rondos-spurny-vector-spaces}.

Also note that if $x \in K$, then each $\mu \in \M(K, \E)$ can be uniquely decomposed as $\mu=\psi \ep_x+\nu$, where $\psi \in \E$ and $\nu \in \M(K, \E)$ with $\nu(\{x\})=0$. Indeed, it is enough to denote $\psi=\mu(\{x\})$ and $\nu=\mu|_{K \setminus \{x\}}$, and then 
\[\mu=\mu|_{\{x\}}+\mu|_{K \setminus \{x\}}=\psi\ep_x+\nu.\]
The uniqueness part is easy. Whenever we write a vector measure $\mu \in \M(K, \E)$ in the form $\mu=\psi \ep_x+\nu$, then we tacitly mean that $\psi \in \E$ and $\nu(\{x\})=0$.

Next, for a bounded Borel function $f \colon K \rightarrow E$ and $e \in S_E$, the function $\wh{f} \otimes e\colon \M(K,\E  )\to \er$ is defined as
\[
(\wh{f} \otimes e)(\mu)=\la \mu, f \otimes e \ra,\quad \mu\in \M(K,\E).
\]

The following factorization was proved in \cite[Lemma 2.7]{rondos-spurny-vector-spaces}. Let $\pi\colon \M(K, \E)\to \H^*$ be the restriction mapping, let $x\in K$ be a weak peak point and $e \in S_{E^+}$. Then there exists $a_{x, e}^{**}\in \H^{**}$ such that \[\la a_{x,e}^{**}, \pi(\mu) \ra=(\wh{\chi_{\{x\}}} \otimes e)(\mu)=\mu(\chi_{\{x\}} \otimes e)\] for any measure $\mu \in \M(K, \E)$ carried by $\ov{\Ch_{\H} K}$. 
Also, if $x_1$ and $x_2$ are distinct weak peak points in $K$, $e_1, e_2 \in S_{E^+}$ and $\alpha_1, \alpha_2 \in \er$ are arbitrary, then
\begin{equation}
\label{norm}
\norm{\alpha_1 a_{x_1,e_1}^{\ast\ast}+\alpha_2 a_{x_2,e_2}^{\ast\ast}} = \max \lbrace \abs{\alpha_1}, \abs{\alpha_2} \rbrace.
\end{equation}

By \cite[Lemma 2.3 and Lemma 2.8(b)]{rondos-spurny-vector-spaces} we moreover know that for each weak peak point $x \in \Ch_{\H} K$ and $e \in S_{E^+}$, the element $a_{x, e}^{**}$ is of the first Borel class on $(rB_{\H^*}, w^*)$ for any $r>0$.

\section{Positive isomorphisms}

In this section we assume that for $i=1,2$, $\H_i$ is a closed subspace of $\C(K_i, E_i)$ for some compact space $K_i$ and a real Banach lattice $E_i$. Further we assume that each point of the Choquet boundary $\Ch_{\H_i} K_i$ of $\H_i$ is a weak peak point and let $S\colon \H_1\to \H_2$ be an isomorphism  mapping the set of positive elements of $\H_1$ into the set of positive elements of $\H_2$. Further, for $i=1,2$, let $\A_i$ be the canonical scalar function space of $\H_i$, let $\pi_i\colon \M(K_i,\E_i)\to \H_i^{*}$ be the restriction mapping and let $\phi_i\colon K_i \times \E_i \to \H_i^*$ be the evaluation mapping. Now we prove results valid in this setting which we then apply in the proof of the Theorem \ref{main} to isomorphisms $T$ and $T^{-1}$. 

We start with the fact that for each $\s \in B_{\E_1}$, $\t \in B_{\E_2}$, $x \in K_1$ and $y \in K_2$ it holds that
\begin{equation}
\label{projekce}
\pi_1(\s\ep_x)=\phi_1(x, \s)\quad\text{in } \H_1^*\quad \text{and}\quad \pi_2(\t\ep_y)=\phi_2(y, \t) \quad\text{in } \H_2^*,
\end{equation}
see \cite[(3.1)]{rondos-spurny-vector-spaces}.

Next, for each $x \in \Ch_{\H_1} K_1$ and $u \in S_{E_1^+}$, we consider the element $a_{x, u}^{**}\in \H_1^{**}$ satisfying
\[\la a_{x, u}^{**}, \pi_1(\mu) \ra=(\wh{\chi_{\{x\}}} \otimes u)(\mu)\]
 for $\mu$ carried by $\ov{\Ch_{\H_1} K_1}$, and we proceed to the following equalities. Let $s \in \H_2^*$, and suppose that $\mu \in \pi_1^{-1}(S^*(s))$ is a Hahn-Banach extension of $S^*(s)$ carried by $\ov{\Ch_{\H_1} K_1}$ written in the form $\mu=\psi \ep_{x}+\nu$. Then we have
\begin{equation}
\nonumber
\begin{aligned}
&\langle S^{\ast\ast}(a_{x,u}^{\ast\ast}), s \rangle_{\H_2^{\ast\ast},\H_2^{\ast}}=\langle a_{x,u}^{\ast\ast}, S^\ast(s) \rangle_{\H_1^{**},\H_1^{\ast}}
=\\&=
\langle a_{x,u}^{\ast\ast}, \pi_1(\mu) \rangle_{\H_1^{**},\H_1^{\ast}}=
\langle \wh{\chi_{\{x\}}} \otimes u, \mu \rangle_{\C(K_1, E_1)^{\ast\ast},\M(K_1, E_1^*)} 
=\\&=
\langle \wh{\chi_{\{x\}}} \otimes u, \psi\ep_x+\nu \rangle_{\C(K_1, E_1)^{\ast\ast},\M(K_1, E_1^*)} =^{\eqref{aplikace}}\langle \psi, u \rangle_{E_1^*, E_1}=
\langle \mu(\{x\}), u \rangle_{E_1^*, E_1}.
\end{aligned}
\end{equation}

Thus using the above notation, we have
\begin{equation}
\label{duality}
\begin{aligned}
\langle S^{\ast\ast}(a_{x,u}^{\ast\ast}), s \rangle_{\H_2^{\ast\ast},\H_2^{\ast}}=\langle \psi, u \rangle_{\E_1, E_1}=\langle \mu(\{x\}), u \rangle_{\E_1, E_1}.
\end{aligned}
\end{equation}

Moreover, for any function $h \in \A_1$ satisfying that $h \otimes u \in \H_1$ we have 
 \begin{equation}
\label{2duality}
\begin{aligned}
\la s, S(h \otimes u) \ra=
\la S^*(s), h \otimes u \ra=\mu(h \otimes u).
\end{aligned}
\end{equation}

The general strategy of the proof of Theorem \ref{main} will be very similar to the one of \cite[Theorem 1.1]{rondos-spurny-vector-spaces}, but we need to make some adjustments to control the positivity of the elements considered. In \cite{rondos-spurny-vector-spaces}, we used as an important ingredient of the proof the fact proved in \cite{camberngriem2} that if $E$ is a reflexive Banach space and $K$ is a compact space,  then the space $\C(K, E)^{**}$ is isometrically isomorphic to the space $\C(Z, E_{w})$, where $Z$ is a compact Hausdorff space depending on $K$, and $E_{w}$ denotes $E$ equipped with its weak topology. Here we use a different approach which requires less theory. To achieve this, we need the following lemma about approximation of the element $a_{x, u}^{**}$ by the net of peaking functions for the pair $(x, u) \in \Ch_{\H_1} K_1 \times S_{E_1}$.

\begin{lemma}
	\label{approx}
	Let $x \in \Ch_{\H_1}K_1, u \in S_{E_1^+}$, and let $\{h_{U, \ep}\}$ be a net of peaking functions for the pair $(x, u)$. Then the following assertions hold.
	\begin{itemize}
		\item[(i)] The net $\{h_{U, \ep} \otimes u\}$ converges weak$^*$ to $a_{x, u}^{**}$ in $\H_1^{**}$. 		
		\item[(ii)]
		The net $\{S(h_{U, \ep} \otimes u)\}$ converges weak$^*$ to $S^{**}(a_{x, u}^{**})$ in $\H_2^{**}$.
	\end{itemize}
\end{lemma}

\begin{proof}

(i)	Let $s \in \H_1^*$ be given, and $\mu \in \pi_1^{-1}(s)$ be a Hahn-Banach extension of $s$ carried by $\ov{\Ch_{\H_1} K_1}$.  Note that $\la a_{x,u}^{**}, s \ra=\la \mu(\{x\}), u \ra$ and $\la s, h_{U, \ep} \otimes u \ra=\mu(h_{U, \ep} \otimes u)$. The proof of those equalities is essentially the same as the proof of \eqref{duality} and \eqref{2duality}, just in this case it is simpler, as there is no operator $S$. 
	Next, for a given $\ep_0>0$ we find a closed set $U_0$ containing $x$ and such that $\abs{\la \mu, u \ra}(U_0 \setminus \{x\})<\ep_0$. Then for each closed subset $U \subseteq U_0$ containing $x$ and $0 <\ep \leq \ep_0$ we obtain
		\begin{equation}
		\nonumber
		\begin{aligned}
		& \abs{\la a_{x,u}^{**}-h_{U, \ep} \otimes u, s \ra}=\abs{\la \mu(\{x\}), u \ra-\mu(h_{U, \ep} \otimes u)} 
		\leq \\& \leq
		\abs{\la \mu(\{x\}), u \ra-\int_{\lbrace x \rbrace} h_{U, \ep}~d\la \mu, u \ra}+\abs{\int_{U \setminus \lbrace x \rbrace} h_{U, \ep}~d\la \mu, u \ra}
		+\\&+\abs{\int_{\ov{\Ch K_1} \setminus U} h_{U, \ep}~d\la \mu, u \ra} \leq
		\abs{\la \mu(\{x\}), u \ra (1-h_{U, \ep}(x))}+\int_{U \setminus \lbrace x \rbrace} h_{U, \ep}~d\abs{\la \mu, u \ra}
		+\\&+
		\int_{\ov{\Ch K_1} \setminus U} h_{U, \ep}~d\abs{\la \mu, u \ra}<  \norm{\mu}\ep+\ep+\norm{\mu}\ep= (2\norm{\mu}+1)\ep \leq (2\norm{s}+1)\ep_0.
		\end{aligned}
		\end{equation}
			
        Since (ii) follows immediatelly from (i) and the fact that $S^{**}$ is weak$^*$-weak$^*$ continuous, the proof is finished.
\end{proof}

\begin{lemma}
	\label{principle}
		Let $x \in \Ch_{\H_1} K_1$ and $u \in S_{E_1^+}$ and let $\{h_{U, \ep}\}$ be the net of peaking functions for the pair $(x, u)$. Then
		\[ \norm{S(h_{U, \ep} \otimes u)}=\sup_{y \in \Ch_{\H_2} K_2, \t \in S_{(\E_2)^+}} \abs{\langle \phi_2(y, \t), S(h_{U, \ep} \otimes u) \rangle}\] 
		and
		\begin{equation}
		\nonumber
		\begin{aligned}
		&\norm{ S^{**}(a_{x, u}^{**})}=\sup_{y \in \Ch_{\H_2} K_2, \t \in S_{(\E_2)^+}} \abs{\langle S^{**}(a_{x, u}^{**}), \phi_2(y, \t) \rangle}.
		\end{aligned}
		\end{equation}
Moreover,
		\[\norm{h_{U, \ep} \otimes u)}=\sup_{x \in \Ch_{\H_1} K_1, u^* \in S_{(\E_1)^+}} \abs{\langle \phi_1(x, u^*), h_{U, \ep} \otimes u \rangle}\]
		and
		\begin{equation}
		\nonumber
		\begin{aligned}
		&\norm{a_{x, u}^{**}}=\sup_{x \in \Ch_{\H_1} K_1, u^* \in S_{(\E_1)^+}} \abs{\langle a_{x, u}^{**}, \phi_1(x, u^*) \rangle}.
		\end{aligned}
		\end{equation} 
\end{lemma}

\begin{proof}
    We prove the first two equalities of the statement, the "moreover" part can be proven in a similar (but simpler) way.
    In view of \eqref{max}, it is enough to prove that the function $S(h_{U, \ep} \otimes u)$, when viewed as an element of $\H_2^{**}$, and $S^{**}(a_{x, u}^{**})$, are both positive elements of $\H_2^{**}$ that are of the first Borel class on $B_{\H_2^*}$. 
    
    This is easily checked in the case of $S(h_{U, \ep} \otimes u)$, since this function is continuous on $B_{\H_2^*}$, and we know that for each $U$ and $\ep$, the peaking function $h_{U, \ep} \otimes u$ is positive in $\H_1$. Thus $S(h_{U, \ep} \otimes u) \in \H_2^+$, by the assumption on $S$.
	
	Next, by Lemma \ref{approx}, the element $S^{**}(a_{x, u}^{**})$ is positive, as it is a weak$^*$ limit of positive elements of the form $S(h_{U, \ep} \otimes u)$. 
	Moreover, from \cite[Lemma 2.3 and Lemma 2.8(b)]{rondos-spurny-vector-spaces} we know that $a_{x, u}^{**}$ is of the first Borel class on any ball in $\H_1^*$, in particular on $\lambda^+(E_2)B_{\H_1^*}$. Since $S^*$ is a weak$^*$-weak$^*$ homeomorphism, $S^*(B_{\H_2^*})\subset \lambda^+(E_2)B_{\H_1^*}$ and $S^{**}(a_{x, u}^{**})=a_{x, u}^{**}\circ S^*$, it follows that $S^{**}(a_{x, u}^{**})$ is of the first Borel class on $B_{\H_2^*}$ as well. The proof is finished.
\end{proof}

\section{Proof of Theorem \ref{main}}
The next lemma describes an important property of the parameter $\lambda^+$. For the proof see \cite[Lemma 5.1]{Villamizar_svazy}. In our case, actually, we need to use this lemma only for $r=1$, in which case its conclusion follows easily from the definition.

\begin{lemma}
	\label{lambda}
	Let $E$ be a Banach lattice. Let $r \in \en$ and $\eta>0$ be fixed and suppose that $\lbrace e_i \rbrace_{i=1}^{2^r} \subset E^+$ satisfy $\norm{e_i} \geq \eta$ for each $1 \leq i \leq 2^r$. Then there exist $\lbrace \alpha_i \rbrace_{i=1}^{2^r} \subset \er$ with $\max \lbrace \abs{\alpha_i}: 1 \leq i \leq 2^r \rbrace \leq 1$ such that
	\[\norm{\sum_{i=1}^{2^r} \alpha_i e_i} \geq \eta (\lambda^+(E))^r.\]
\end{lemma}

Now we are ready to prove the main result.\\
\emph{Proof of Theorem \ref{main}}

We may assume that the spaces $K_1, K_2$ are compact. Indeed, if $K_1, K_2$ were locally compact then we would consider their one-point compactifications $J_i=K_i\cup\{\alpha_i\}$, where, for $i=1,2$, $\alpha_i$ is the point representing infinity. Then the spaces $\H_i$ are isometric to closed subspaces $\widetilde{\H}_i \subset \C(J_i,\er)$ satisfying $h(\alpha_i)=0$, $h\in \widetilde{\H}_i$, with $\Ch_{\H_i} K_i$ homeomorphic to $\Ch_{\widetilde{\H}_i} J_i$ and all points of $\Ch_{\widetilde{\H}_i} J_i$ are weak peak points by \cite[Lemma 2.10]{rondos-spurny-vector-spaces}. 

Secondly, we suppose that there exists an $\ep>0$ such that $\Vert Tf \Vert \geq (1+\ep)\Vert f \Vert$ for $f \in \H_1$ and $\Vert T \Vert < \min \lbrace \lambda^+(E_1), \lambda^+(E_2) \rbrace$ (otherwise we replace $T$ by the isomorphism $(1+\ep)\norm{T^{-1}}T$). We fix $P$ such that $1<P<1+\ep$. Hence $T$ satisfies $\Vert Tf \Vert > P\norm{f}$ for $f \in \H_1, f \neq 0$. 

\emph{Claim 1.: For any $a^{**}\in \H_1^{**}\setminus \{0\}$ and $b^{**}\in \H_2^{**}\setminus \{0\}$ we have $\norm{T^{**}(a^{**})}>P\norm{a^{**}}$ and $\norm{(T^{**})^{-1}(b^{**})}>\frac{1}{\min \lbrace \lambda^+(E_1), \lambda^+(E_2) \rbrace}\norm{b^{**}}$.}

This follows from the fact that, for any operator $T$, $\norm{T}=\norm{T^{**}}$.

For each $x\in \Ch_{\H_1} K_1$ and $u \in S_{E_1^+}$, we consider the element $a_{x, u}^{**}\in \H_1^{**}$ satisfying
\[\la a_{x, u}^{**}, \pi_1(\mu) \ra=(\wh{\chi_{\{x\}}} \otimes u)(\mu)\]
for $\mu$ carried by $\ov{\Ch_{\H_1} K_1}$.
Analogously we define for $y\in\Ch_{\H_2} K_2$ and $v \in S_{E_2^+}$ the element $b_{y, v}^{**}\in \H_2^{**}$.  

\begin{definition}	
For $x \in \Ch_{\H_1} K_1$ and $y \in \Ch_{\H_2} K_2$ we define relations $\rho_1$ and $\rho_2$ as follows:
\begin{equation}
\label{rho}
\nonumber
\begin{aligned}
\rho_1(x)=&\{y\in \Ch_{\H_2} K_2, \exists v \in S_{E_2^+},\exists \s \in S_{(\E_1)^+} \colon \\&
\abs{ \langle (T^{\ast\ast})^{-1}(b_{y, v}^{\ast\ast}), \phi_1(x, \s) \rangle} >\frac{1}{\min \lbrace \lambda^+(E_1), \lambda^+(E_2) \rbrace} \},\\
\rho_2(y)=&\left\{x\in \Ch_{\H_1} K_1, \exists u \in S_{E_1^+},\exists \t \in S_{(\E_2)^+}: \abs{\langle T^{\ast\ast}(a_{x, u}^{\ast\ast}), \phi_2(y, \t) \rangle} >P\right\}.\\
\end{aligned}
\end{equation}
\end{definition}

In the rest of the proof we show that $\rho_1$ is the desired homeomorphism from $\Ch_{\H_1} K_1$ to $\Ch_{\H_2} K_2$, with $\rho_2$ being its inverse. 

First note that we have the following equivalent descriptions of the relations $\rho_1$ and $\rho_2$.
\begin{lemma}
	\label{ekvi}
	Let $x \in \Ch_{\H_1} K_1, y \in \Ch_{\H_2} K_2$. Then the following assertions hold.
	\begin{itemize}
		\item[(i)] $x \in \rho_2(y)$ if and only if there exists $u \in S_{E_1^+}$  and $v^* \in S_{(E_2^*)^+}$ such that the net of peaking functions $\{h_{U, \ep}\}$ for the pair $(x, u)$ satisfies 
		\[\lim_{U, \ep} \abs{\la \phi_2(y, v^*), T(h_{U, \ep} \otimes u) \ra}>P\]
		and this happens if and only if there exist points $u \in S_{E_1^+}$  and $v^* \in S_{(E_2^*)^+}$ such that whenever 
		\[\mu \in \pi_1^{-1}(T^\ast(\phi_2(y, \t))) \cap \M(K_1, \E_1)\]
		 is a Hahn-Banach extension of $T^\ast(\phi_2(y, \t))$ which is carried by $\ov{\Ch_{\H_1} K_1}$, then $\abs{\langle \mu(\{x\}), u \rangle }>P$.
		 \item[(ii)] $y \in \rho_1(x)$ if and only if there exists $v \in S_{E_2^+}$  and $u^* \in S_{(E_1^*)^+}$ such that the net of peaking functions $\{h_{U, \ep}\}$ for the pair $(y, v)$ satisfies 
		 \[\lim_{U, \ep} \abs{\la \phi_1(x, u^*), T^{-1}(h_{U, \ep} \otimes v) \ra}>(\min \lbrace \lambda^+(E_1), \lambda^+(E_2) \rbrace)^{-1}\]
		 and this happens if and only if there exist points $v \in S_{E_2^+}$  and $u^* \in S_{(E_1^*)^+}$ such that whenever 
		\[\mu \in \pi_2^{-1}((T^\ast)^{-1}(\phi_1(x, \s))) \cap \M(K_2, \E_2)\] is a Hahn-Banach extension of $(T^\ast)^{-1}(\phi_1(x, \s))$ which is carried by $\ov{\Ch_{\H_2} K_2}$, then $\abs{\langle \mu(\{y\}), v \rangle}>(\min \lbrace \lambda^+(E_1), \lambda^+(E_2) \rbrace)^{-1}$.
	\end{itemize}
	\end{lemma}
\begin{proof}
We prove (i), the proof of (ii) can be done in the same way. 

By Lemma \ref{approx} we know that that the net $\{\abs{\la \phi_2(y, v^*), T(h_{U, \ep} \otimes u) \ra} \}_{U, \ep}$ converges to $\abs{\langle T^{\ast\ast}(a_{x, u}^{\ast\ast}), \phi_2(y, \t) \rangle}$, which proves the first part. The rest of (i) follows from the fact that, for an arbitrary $\mu \in \pi_1^{-1}(T^\ast(\phi_2(y, \t)))$, a Hahn-Banach extension of $T^\ast(\phi_2(y, \t))$ carried by $\ov{\Ch_{\H_1} K_1}$ (and at least one such measure exists by \cite[Lemma 2.2]{rondos-spurny-vector-spaces}) holds by \eqref{duality} that

\[\langle T^{\ast\ast}(a_{x,u}^{\ast\ast}), \phi_2(y, \t) \rangle_{\H_2^{\ast\ast},\H_2^{\ast}}=\langle \mu(\{x\}), u \rangle_{\E_1, E_1}.\]
\end{proof}

\emph{Claim 2. $\rho_1$ and $\rho_2$ are mappings.}

We show that $\rho_2(y)$ is at most single-valued for each $y\in \Ch_{\H_2} K_2$. Suppose that there are distinct $x_1, x_2 \in \Ch_{\H_1} K_1$ such that $x_i \in \rho_2(y)$ for $i=1, 2$. Thus there exist points $v_i^* \in S_{(\E_2)^+}$ and $u_i \in S_{E_1^+}$ such that 
\[\abs{\langle T^{\ast\ast}(a_{x_i, u_i}^{\ast\ast}), \phi_2(y, \t_i) \rangle} >P.\]
Let $\ep_0>0$ satisfy that for $i=1, 2$, 
\[\abs{\langle T^{\ast\ast}(a_{x_i, u_i}^{\ast\ast}), \phi_2(y, \t_i) \rangle} >(1+\ep_0)P.\]

By Lemma~\ref{ekvi} there exists $\ep \leq \ep_0$, and for $i=1, 2$, there exist closed disjoint sets $U_i$, each containing $x_i$, and peaking functions $h_{U_i, \ep}$ for the pairs $(x_i, u_i)$ such that $\abs{\la \phi_2(y, v_i^*), T(h_{U_i, \ep} \otimes u_i) \ra}>(1+\ep_0)P$.
Thus for $i=1, 2$ we have

 \begin{equation}
\nonumber
\begin{aligned}
&\norm{T(h_{U_i, \ep} \otimes u_i)(y)}_{E_2} \geq \abs{\la \t_i, T(h_{U_i, \ep} \otimes u_i)(y) \ra}
=\\&=
\abs{\la \phi_2(y, \t_i), T(h_{U_i, \ep} \otimes u_i) \ra}>(1+\ep_0)P.
\end{aligned}
\end{equation}

Then, since $T(h_{U_i, \ep} \otimes u_i)(y) \geq 0$ in $E_2$ for $i=1, 2$, by Lemma \ref{lambda} there exist $\alpha_1, \alpha_2 \in \er$ with $\vert \alpha_i \vert \leq 1$ for $i=1, 2$, such that
\[ \norm{\alpha_1T(h_{U_1, \ep} \otimes u_1)(y)+\alpha_2T(h_{U_2, \ep} \otimes u_2)(y)}_{E_2} \geq (1+\ep_0)P \lambda^+(E_2).\]
Thus 
\begin{equation}
\nonumber
\begin{aligned}
&\norm{T(\alpha_1 (h_{U_1, \ep} \otimes u_1)+\alpha_2(h_{U_2, \ep} \otimes u_2))}_{\sup} 
\geq \\& \geq
\norm{\alpha_1T(h_{U_1, \ep} \otimes u_1)(y)+\alpha_2T(h_{U_2, \ep} \otimes u_2)(y)}_{E_2} 
\geq \\& \geq 
(1+\ep_0)P \lambda^+(E_2) > (1+\ep_0)\lambda^+(E_2).
\end{aligned}
\end{equation}

On the other hand, by Lemma \ref{principle} we obtain that 
\begin{equation}
\nonumber
\begin{aligned}
&\norm{\alpha_1(h_{U_1, \ep} \otimes u_1)+\alpha_2(h_{U_2, \ep} \otimes u_2)} 
\leq \\& \leq
\sup_{x \in \Ch_{\H_1} K_1, \s \in B_{\E_1}} \abs{ \alpha_1} \abs{h_{U_1, \ep}(x)} \abs{\la \s, u_1 \ra}+\abs{\alpha_2} \abs{h_{U_2, \ep}(x)} \abs{\la \s, u_2 \ra} 
\leq \\& \leq 1+\ep \leq 1+\ep_0.
\end{aligned}
\end{equation}
 Thus we obtained a contradiction with $\norm{T}<\min \lbrace \lambda^+(E_1), \lambda^+(E_2) \rbrace \leq \lambda^+(E_2)$,  and hence $\rho_2$ is a mapping.
Analogously we would show that $\rho_1(x)$ is at most single-valued for each $x\in \Ch_{\H_1} K_1$.

Next we use Lemma \ref{principle} to check that the mappings $\rho_1$ and $\rho_2$ are surjective.
Let $L_1$ and $L_2$ denote the domain of $\rho_1$ and $\rho_2$, respectively.

\emph{Claim 3.: The mappings $\rho_1\colon L_1\to \Ch_{\H_2} K_2$ and $\rho_2\colon L_2\to \Ch_{\H_1} K_1$ are surjective.}
Let $x\in \Ch_{\H_1} K_1$ be given and choose arbitrary $u \in S_{E_1^+}$. By \eqref{norm} we know that $\norm{a_{x, u}^{**}}=1$. Thus by Claim 1 and Lemma \ref{principle} we have

\[
\begin{aligned}
&P < \Vert T^{**}(a_{x, u}^{**}) \Vert = \sup_{y \in \Ch_{\H_2} K_2, \t \in S_{(\E_2)^+}} \abs{\langle T^{**}(a_{x, u}^{**}), \phi(y, \t) \rangle}.
\end{aligned}
\]
Thus there exist $y \in \Ch_{\H_2} K_2$ and $\t \in S_{(\E_2)^+}$ such that $P<\abs{\langle T^{**}(a_{x, u}^{**}), \phi(y, \t) \rangle}$, that is, $\rho_2(y)=x$.
Analogously we would check that $\rho_1$ is surjective.

\emph{Claim 4.: We have $L_1=\Ch_{\H_1} K_1$ and $L_2=\Ch_{\H_2} K_2$ and $\rho_2(\rho_1(x))=x$, $x\in\Ch_{\H_1} K_1$, and $\rho_1(\rho_2(y))=y$, $y\in\Ch_{\H_2} K_2$.}

Suppose that $y \in \Ch_{\H_2} K_2$, $\rho_2(y)=x$, but $x \notin L_1$ or $\rho_1(x) \neq y$. In both cases we obtain that for all $v \in S_{E_2^+}$ and $u^* \in S_{(\E_1)^+}$, 
\[\abs{\langle (T^{\ast\ast})^{-1}(b_{y, v}^{\ast\ast}), \phi_1(x, \s) \rangle}\leq (\min \lbrace \lambda^+(E_1), \lambda^+(E_2) \rbrace)^{-1}.\]
 For $v \in S_{(E_2)^+}$ we denote 
\[ Q_v=\sup_{\tilde{x} \in \Ch_{\H_1} K_1, u^* \in S_{(\E_1)^+}} \abs{\langle (T^{\ast\ast})^{-1}(b_{y, v}^{\ast\ast}), \phi_1(\tilde{x}, \s) \rangle}=^{\text{Lemma } \ref{principle}}\norm{(T^{**})^{-1}(b_{y,v}^{**})}\]
and 
\[Q=\sup_{v \in S_{E_2^+}} Q_v.\]

We know that $\rho_1$ is surjective. This means that \[Q >(\min \lbrace \lambda^+(E_1), \lambda^+(E_2) \rbrace)^{-1}.\] Let $\ep>0$ satisfy 

\begin{equation}
\nonumber
\begin{aligned}
&\ep<\frac{2P-\min \lbrace \lambda^+(E_1), \lambda^+(E_2) \rbrace}{\min \lbrace \lambda^+(E_1), \lambda^+(E_2) \rbrace P}
\\& \text{and} \quad Q -\ep> (\min \lbrace \lambda^+(E_1), \lambda^+(E_2) \rbrace)^{-1}.
\end{aligned}
\end{equation}

By the definition of $Q$, let $v \in S_{E_2^+}$, $u^* \in S_{(E_1^*)^+}$ and $\tilde{x} \in \Ch_{\H_1} K_1$ be such that the vector 
$u_1=(T^{\ast\ast})^{-1}(b_{y, v}^{\ast\ast})(\tilde{x}) \in E_1^{**} \simeq E_1$ defined by
\[\la (T^{\ast\ast})^{-1}(b_{y, v}^{\ast\ast})(\tilde{x}), \tilde{u}^*\ra=\langle (T^{\ast\ast})^{-1}(b_{y, v}^{\ast\ast}), \phi_1(\tilde{x}, \tilde{u}^*) \rangle, \quad \tilde{u}^* \in \E_1,\]
satisfies
\[\norm{u_1} \geq \abs{\langle (T^{\ast\ast})^{-1}(b_{y, v}^{\ast\ast}), \phi_1(\tilde{x}, \s) \rangle} \geq Q-\ep.\]
From the proof of Lemma \ref{principle} and \eqref{positiveness} it follows that the vector $u_1$ is positive. We denote $u_2=\frac{u_1}{\norm{u_1}} \in S_{E_1^+}$.
Now we consider the element $T^{**}(a_{\tilde{x}, u_2}^{**})$. Since $\norm{a_{\tilde{x}, u_2}^{**}}=1$ by \eqref{norm}, we know that \[\norm{T^{**}(a_{\tilde{x}, u_2}^{**})} > P.\] This by Lemma \ref{principle} means that there exist $\tilde{y} \in \Ch_{\H_2} K_2$ and $\t \in S_{(\E_2)^+}$ such that
\[\abs{\la T^{**}(a_{\tilde{x}, u_2}^{**}), \phi_2(\tilde{y},\t)\ra}>P.\] 
Hence $\rho_2(\tilde{y})=\tilde{x}$. Thus $y \neq \tilde{y}$, since $\rho_2(\tilde{y})=\tilde{x}\neq x =\rho_2(y)$. 
Now, if we pick $\mu \in \pi_1^{-1}(T^*(\phi_2(\tilde{y}, \t)))$, a Hahn-Banach extension of $T^*(\phi_2(\tilde{y}, \t))$ carried by $\ov{\Ch_{\H_1} K_1}$ and write it in the form $\mu=\psi \ep_{\tilde{x}}+\nu$, where $\psi \in \E_1$ and $\nu \in \M(\ov{\Ch_{\H_1} K_1}, \E)$ with $\nu(\{\tilde{x}\})=0$, then by \eqref{duality},
\[\abs{\langle \psi, u_2 \rangle}=\abs{\la T^{**}(a_{\tilde{x}, u_2}^{**}), \phi_2(\tilde{y},\t)\ra}>P.\]
Thus $\norm{\psi}>P$ and 
\begin{equation}
\label{pomoc}
\abs{\langle \psi, u_1 \rangle}=\norm{u_1} \abs{\langle \psi, u_2 \rangle}>(Q-\ep) P.
\end{equation}
Notice that if $\pi_1^*:\H_1^{**} \rightarrow \C(K_1, E_1)^{**}$ is the adjoint mapping of the projection $\pi_1$, then it holds that
\begin{equation}
\label{pom}
\begin{aligned}
&\langle \pi_1^*((T^{**})^{-1}(b_{y, v}^{**})), \psi \ep_{\tilde{x}} \rangle=^{\eqref{projekce}}\langle (T^{**})^{-1}(b_{y, v}^{**}), \phi_1(\tilde{x}, \psi) \rangle
=\\&=\langle (T^{**})^{-1}(b_{y, v}^{**})(\tilde{x}), \psi \rangle.
\end{aligned}
\end{equation}
Thus we have
\begin{equation}
\nonumber
\begin{aligned}
  0&=\la \wh{\chi_{y}} \otimes v, \t\ep_{\tilde{y}} \ra_{\C(K_2, E_2)^{**}, \M(K_2, \E_2)}=\la b_{y, v}^{**}, \pi_2(\t \ep_{\tilde{y}}) \ra_{\H_2^{**}, \H_2^*}
  =^{\eqref{projekce}}\\&=
  \langle b_{y, v}^{**}, \phi_2(\tilde{y}, \t) \rangle_{\H_2^{**}, \H_2^*}=
  \langle (T^{**})^{-1}(b_{y, v}^{**}), T^*\phi_2(\tilde{y}, \t) \rangle_{\H_1^{**}, \H_1^*}
  =\\&=
  \langle (T^{**})^{-1}(b_{y, v}^{**}), \pi_1(\mu) \rangle_{\H_1^{**}, \H_1^*}=
  \langle \pi_1^*((T^{**})^{-1}(b_{y, v}^{**})), \mu \rangle_{\C(K_1, E_1)^{**}, \M(K_1, \E_1)}
  =\\&=
  \langle \pi_1^*((T^{**})^{-1}(b_{y, v}^{**})), \psi \ep_{\tilde{x}}+\nu \rangle_{\C(K_1, E_1)^{**}, \M(K_1, \E_1)}
  =^{\eqref{pom}}\\&=
  \langle \psi,(T^{**})^{-1}(b_{y, v}^{**})(\tilde{x}) \rangle_{\E_1, E_1}+\langle \pi_1^*((T^{**})^{-1}(b_{y, v}^{**})),\nu \rangle_{\C(K_1, E_1)^{**}, \M(K_1, \E_1)}
  =\\&=
  \langle \psi, u_1 \rangle_{\E_1, E_1}+ \langle \pi_1^*((T^{**})^{-1}(b_{y, v}^{**})),\nu \rangle_{\C(K_1, E_1)^{**}, \M(K_1, \E_1)}.
\end{aligned}
\end{equation}
Hence
\[\abs{\langle \psi, u_1 \rangle}=\abs{\langle \pi_1^*((T^{**})^{-1}(b_{y, v}^{**})),\nu \rangle}.\]
On the other hand, we know that $\norm{\nu} \leq \norm{\mu}-\norm{\psi}<\min \lbrace \lambda^+(E_1), \lambda^+(E_2) \rbrace-P$, and thus

\begin{equation}
\nonumber
\begin{aligned}
&\abs{\langle \pi_1^*((T^{**})^{-1}(b_{y, v}^{**})),\nu \rangle} \leq \norm{\pi_1^*((T^{**})^{-1}(b_{y, v}^{**}))}\norm{\nu}
\leq \norm{(T^{**})^{-1}(b_{y, v}^{**})}(\norm{\mu}-\norm{\psi})
<\\&<
Q_{v}(\min \lbrace \lambda^+(E_1), \lambda^+(E_2) \rbrace-P) \leq Q(\min \lbrace \lambda^+(E_1), \lambda^+(E_2) \rbrace-P).
\end{aligned}
\end{equation}

Thus using \eqref{pomoc} we deduce that $(Q -\ep)P \leq Q(\min \lbrace \lambda^+(E_1), \lambda^+(E_2) \rbrace-P)$, that is, 
\[\ep \geq \frac{Q(2P-\min \lbrace \lambda^+(E_1), \lambda^+(E_2) \rbrace)}{P} \geq \frac{2P-\min \lbrace \lambda^+(E_1), \lambda^+(E_2) \rbrace}{\min \lbrace \lambda^+(E_1), \lambda^+(E_2) \rbrace P}.\]
This contradicts the choice of $\ep$ and shows that $x \in L_1$ and $\rho_1(x)=y$.

Now, let $x\in\Ch_{\H_1} K_1$ be given. Then there exists $y\in L_2$ such that $\rho_2(y)=x$. Then $y=\rho_1(\rho_2(y))=\rho_1(x)$, which means that $x\in L_1$.

Let $y\in\Ch_{\H_2} K_2$ be given. Then we can find $x\in L_1=\Ch_{\H_1} K_1$ with $\rho_1(x)=y$ and further we can select $\wh{y}\in L_2$ such that $\rho_2(\wh{y})=x$. Then
\[
y=\rho_1(x)=\rho_1(\rho_2(\wh{y}))=\wh{y}\in L_2.
\]
Hence $L_2=\Ch_{\H_2} K_2$.

Finally, if $x\in \Ch_{\H_1} K_1$, we find $y\in \Ch_{\H_2} K_2$ with $\rho_2(y)=x$ and obtain
\[
\rho_2(\rho_1(x))=\rho_2(\rho_1(\rho_2(y)))=\rho_2(y)=x.
\]

\begin{remark}
\label{remark}
The fact that in the previous claim, the element   $(T^{\ast\ast})^{-1}(b_{y, v}^{\ast\ast})(\tilde{x}) \in E_1^{**}$ actually belongs to $E_1$ is the only part of the proof where the reflexivity of $E_1$ is used. In the case where $\H_2=\C(K_2, E_2)$, the reflexivity is not needed since the above fact already follows from the assumption $\lambda^+(E_1)>1$. Indeed, since it is easy to check that $\lambda^+(c_0)=1$, it follows that $E_1$ does not contain an isomorphic copy of $c_0$. Thus by \cite[Theorem 4.4]{BrooksLewis} and \cite[Theorem 13]{Panchapagesan} (see also \cite[pages 4 nad 5]{Villamizar_svazy}), each operator $S:\C(K_2, E_2) \rightarrow E_1$ is represented by a Borel measure $\mu$ on $K_2$ taking values in $L(E_2, E_1)$. The measure satisfies that for each fixed $v \in E_2$ and $u^* \in \E_1$, the scalar measure 
\begin{equation}
\label{measure}
\mu^{v, u^*}(f)=\la u^*, \int (f \otimes v) d\mu \ra=\la u^*, S(f \otimes v \ra), \quad f \in \C(K_2, \er),
\end{equation}
belongs to $\M(K_2, \er)$ (see \cite[Chapter 5]{Dinculeanu} for the definition of integration with respect to $\mu$). 

In our case, suppose that we have $y, v$ and $\tilde{x}$ as in the proof of Claim 4. We denote $T^{-1}_{\tilde{x}}: \C(K_2, E_2) \rightarrow E_1$ defined by $T^{-1}_{\tilde{x}}(f)=T^{-1}(f)(\tilde{x})$ for $f \in \C(K_2, E_2)$. From above, to this operator corresponds a measure $\mu_{\tilde{x}}$ defined on Borel subsets of $K_2$ and taking values in $L(E_2, E_1)$. Now we pick $u^* \in E_1^*$, and let $\{h_{U, \ep}\}$ be the net of peaking functions for the pair $(y, v)$. Then
\begin{equation}
\nonumber
\begin{aligned}
&\la (T^{\ast\ast})^{-1}(b_{y, v}^{\ast\ast})(\tilde{x}), u^* \ra=\la (T^{\ast\ast})^{-1}(b_{y, v}^{\ast\ast}), \phi_1(\tilde{x}, u^*) \ra
=^{\text{Lemma } \ref{approx}}\\&=\lim_{U, \ep} \la T^{-1} (h_{U, \ep} \otimes v), \phi_1(\tilde{x}, u^*) \ra=\lim_{U, \ep} \la u^*, T^{-1} (h_{U, \ep} \otimes v) (\tilde{x}) \ra
=\\&=
\lim_{U, \ep} \la u^*, T^{-1}_{\tilde{x} } (h_{U, \ep} \otimes v) \ra=^{\eqref{measure}}
\lim_{U, \ep} \mu_{\tilde{x}}^{v, u^*} (h_{U, \ep})=\\&=
\mu_{\tilde{x}}^{v, u^*}(\{ y \})=\la u^*, \mu_{\tilde{x}}(\{ y \})(v)\ra.
\end{aligned}
\end{equation}
Thus
\[(T^{\ast\ast})^{-1}(b_{y, v}^{\ast\ast})(\tilde{x})=\mu_{\tilde{x}}(\{ y \})(v) \in E_1.\]
It would be interesting to know for which spaces $\H_2 \subseteq \C(K_2, E_2)$, where $E_1$ need not be reflexive, but $\lambda^+(E_1)>1$, the element $(T^{\ast\ast})^{-1}(b_{y, v}^{\ast\ast})(\tilde{x})$ always belongs to $E_1$.
\end{remark}

Till now we have proved that $\rho_1\colon \Ch_{\H_1} K_1\to \Ch_{\H_2} K_2$ is a bijection with $\rho_2$ being its inverse.
Now we check that $\rho_1$ is a homeomorphism. To this end, note that the definition of the mappings $\rho_1$ and $\rho_2$ may be now be rewritten in the following way:
\begin{equation}
\label{prepis}
\begin{aligned}
\rho_1(x)=&\{y\in \Ch_{\H_2} K_2, \forall v \in S_{E_2^+} \exists \s \in S_{(\E_1)^+} \colon\\&
 \abs{\langle (T^{\ast\ast})^{-1}(b_{y, v}^{\ast\ast}), \phi_1(x, \s) \rangle} >\frac{1}{\min \lbrace \lambda^+(E_1), \lambda^+(E_2) \rbrace} \}, \quad x \in \Ch_{\H_1} K_1,\\
\rho_2(y)&=\{x\in\Ch_{\H_1} K_1, \forall u \in S_{E_1^+} \exists \t \in S_{(\E_2)^+}\colon\\&
\abs{\langle T^{\ast\ast}(a_{x, u}^{\ast\ast}), \phi_2(y, \t) \rangle} >P\}, \quad y \in \Ch_{\H_2} K_2.\\
\end{aligned}
\end{equation}
We show that the formula above holds for $\rho_2$, the proof for $\rho_1$ is similar. Suppose that $y \in \Ch_{\H_2} K_2$ and $\rho_2(y)=x$. If $u \in S_{E_1^+}$ is arbitrary, then by Claim 1 and Lemma \ref{principle} we obtain that
\[
\begin{aligned}
&P < \Vert T^{**}(a_{x, u}^{**}) \Vert = \sup_{\tilde{y} \in \Ch_{\H_2} K_2, \t \in S_{(\E_2)^+}} \abs{\langle T^{**}(a_{x, u}^{**}), \phi_2(\tilde{y}, \t) \rangle}.
\end{aligned}
\]
Thus there exist $\tilde{y} \in \Ch_{\H_2} K_2$ and $\t \in S_{(\E_2)^+}$ such that $\abs{\langle T^{**}(a_{x, u}^{**}), \phi_2(\tilde{y}, \t) \rangle}>P$, that is, $\rho_2(\tilde{y})=x$. But since we know that $\rho_2$ is a bijection, this means that $y=\tilde{y}$. 

\emph{Claim 5.: The mapping $\rho_2$ is continuous.}

Assuming the contrary, there exists a net $\lbrace y_\beta: \beta \in B \rbrace \subset \Ch_{\H_2} K_2$ such that $y_\beta \rightarrow y_0 \in \Ch_{\H_2} K_2$ but $x_\beta=\rho_2(y_\beta) \nrightarrow \rho_2(y_0)=x_0$. Then there exists a closed neighbourhood $V$ of $x_0$ such that for each $\beta_0 \in B$ there exists $\beta \geq \beta_0$ such that $x_\beta \notin V$. 

Fix a point $u \in S_{E_1^+}$. Since $\rho_2(y_0)=x_0$, by \eqref{prepis} and \eqref{duality} there exists $\t_0 \in S_{(\E_2)^+}$ such that whenever $\mu_0 \in \pi_1^{-1}(T^*(\phi_2(y_0, \t_0)))$ is a Hahn-Banach extension of $T^*(\phi_2(y_0, \t_0))$ carried by $\ov{\Ch_{\H_1} K_1}$, then $\abs{\la \mu_0(\lbrace x_0 \rbrace), u \ra}>P$. We pick such a  $\mu_0$ and write it in the form $\mu_0=\psi_0 \ep_{x_0}+\nu_0$, where $\psi_0 \in \E_1$ and $\nu_0 \in \M(\ov{\Ch_{\H_1} K_1}, \E_1)$ with $\nu_0(\lbrace x_0 \rbrace)=0$. Then $\abs{\langle \psi_0, u \rangle}=\abs{\la \mu_0(\lbrace x_0 \rbrace), u \ra}>P$.

Now, choose $\ep \in (0, 1)$ such that $ \frac{1+3\ep}{1-\ep}<P$. Then, since \[\norm{\mu_0} \leq \min \lbrace \lambda^+(E_1), \lambda^+(E_2) \rbrace \leq 2,\] we have
\[\frac{1+\ep (\norm{\mu_0}+1)}{1-\ep}\leq \frac {1+3\ep}{1-\ep}<P,\] 
and we may choose a closed neighbourhood $V_1$ of $x_0$ such that $V_1 \subset V$ and 
\[\vert \langle \nu_0, u \rangle \vert(V_1)<P(1-\ep)-(1+\ep (\Vert \mu_0 \Vert+1)).\]
Let $h_{V_1, \ep}$ be a peaking function for the pair $(x_0, u)$. Then we have 
\begin{equation}
\nonumber
\begin{aligned}
&\abs{\langle \t_0, T(h_{V_1, \ep} \otimes u)(y_0) \rangle}=\abs{\langle \t_0\ep_{y_0}, T(h_{V_1, \ep} \otimes u) \rangle}
=\\&=^{\eqref{projekce}}\abs{\langle \phi_2(y_0, \t_0), T(h_{V_1, \ep} \otimes u)\rangle}=
\abs{\langle T^*(\phi_2(y_0, \t_0)), h_{V_1, \ep} \otimes u \rangle}
=\\&=
\abs{\langle \mu_0, h_{V_1, \ep} \otimes u \rangle}=\abs{\langle \psi_0\ep_{x_0} +\nu_0, h_{V_1, \ep} \otimes u \rangle}
=\\&=
\abs{h_{V_1, \ep}(x_0)\langle \psi_0, u \rangle+ \int_{V_1} h_{V_1, \ep}~d\langle \nu_0, u \rangle+ \int_{\ov{\Ch_{\H_1} K_1} \setminus V_1} h_{V_1, \ep}~ d\langle \nu_0, u \rangle}
\geq \\& \geq
h_{V_1, \ep}(x_0)\abs{\la \psi_0, u \ra}-\int_{V_1} \abs{h_{V_1, \ep}} ~d\abs{\langle \nu_0, u \rangle}-\int_{\ov{\Ch_{\H_1} K_1} \setminus V_1} \abs{h_{V_1, \ep}}~ d\abs{\langle \nu_0, u \rangle}
> \\& >
(1-\ep)P-(P(1-\ep)-(1+\ep \Vert \mu_0 \Vert+\ep))-\ep\norm{\mu_0} = 1 + \ep.
\end{aligned}
\end{equation}
Thus $\norm{T(h_{V_1, \ep} \otimes u)(y_0)}>1+\ep$. Since $y_\beta \rightarrow y_0$ and $T(h_{V_1, \ep} \otimes u)$ is continuous, there exists a $\beta_0 \in B$ such that for all $\beta \geq \beta_0$ we have $\norm{T(h_0 \otimes u)(y_\beta)}>1+\ep$. Thus we can fix a $\beta \in B$ satisfying that $\norm{T(h_0 \otimes u)(y_\beta)}>1+\ep$ and $x_\beta=\rho_2(y_\beta) \notin V$.

 Then again by \eqref{prepis} and \eqref{duality} there exists $\t_\beta \in S_{(\E_2)^+}$ such that whenever $\mu_\beta \in \pi_1^{-1}(T^*(\phi_2(y_\beta, \t_\beta)))$ is a Hahn-Banach extension of $T^*(\phi_2(y_\beta, \t_\beta))$ carried by $\ov{\Ch_{\H_1} K_1}$, then $\abs{\langle \mu_\beta(\lbrace x_\beta \rbrace), u \rangle}> P$. We pick such a  $\mu_\beta$ and write it in the form $\mu_\beta=\psi_\beta \ep_{x_\beta}+\nu_\beta$, where $\psi_\beta \in \E$ and $\nu_\beta \in \M(\ov{\Ch_{\H_1} K_1}, \E)$ with $\nu_\beta(\lbrace x_\beta \rbrace)=0$. Then $\abs{\langle \psi_\beta, u \rangle}=\abs{\langle \mu_\beta(\lbrace x_\beta \rbrace), u \rangle}>P$. 
Next, from the choice of $\ep$ it follows as above that $P(1-\ep)-(1+\ep (\Vert \mu_{\beta} \Vert+1))>0$, so we can choose a closed neighbourhood $V_2$ of $x_\beta$ disjoint from $V$ such that
\[\vert \langle \nu_\beta, u \rangle \vert(V_2)<P(1-\ep)-(1+\ep (\Vert \mu_{\beta} \Vert+1)).\]
Let $h_{V_2, \ep}$ be a peaking function for the pair $(x_{\beta}, u)$. Then as above we obtain that $\norm{T(h_\beta \otimes u)(y_\beta)}>1+\ep$. Now, by the positivity of both $T(h_0 \otimes u)(y_\beta)$ and $T(h_\beta \otimes u)(y_\beta)$, and by Lemma \ref{lambda} there exist $\alpha_1, \alpha_2 \in \er$ such that $\abs{\alpha_1} \leq 1, \abs{\alpha_2} \leq 1$ and 
\begin{equation}
\nonumber
\begin{aligned}
&\norm{T(\alpha_1 (h_0 \otimes u)+\alpha_2(h_\beta \otimes u))}_{\sup} \geq \norm{\alpha_1 T(h_0 \otimes u)(y_\beta)+\alpha_2 T(h_\beta \otimes u)(y_\beta)}
>\\&>
(1+\ep)\lambda^+(E_2). 
\end{aligned}
\end{equation}
On the other hand, by Lemma \ref{approx} we have 
\begin{equation}
\nonumber
\begin{aligned}
&\norm{\alpha_1(h_0 \otimes u)+\alpha_2(h_\beta \otimes u)}_{\sup}=\\&= \sup_{x \in \Ch_{\H_1} K_1, u^* \in S_{(\E_1)^+}} \vert \la u^*, u \ra(\alpha_1h_0(x)+\alpha_2 h_\beta(x)) \vert
\leq \\& \leq
\sup_{x \in \Ch_{\H_1} K_1} \vert \alpha_1h_0(x)+\alpha_2 h_\beta(x) \vert
\leq 1+\ep.
\end{aligned}
\end{equation}
Thus, since $\Vert T \Vert \leq \min \lbrace \lambda^+(E_1), \lambda^+(E_2) \rbrace$ we obtain that 
\[\Vert T(\alpha_1(h_0 \otimes u)+\alpha_2(h_\beta \otimes u)) \Vert \leq \min \lbrace \lambda^+(E_1), \lambda^+(E_2) \rbrace (1+\ep) \leq \lambda^+(E_2)(1+\ep).\] This contradiction proves that $\rho_2$ is continuous. Analogously we would verify that $\rho_1$ is continuous.


\bibliography{iso-functions}\bibliographystyle{siam}
\end{document}